\documentclass[12pt]{amsart}
\usepackage[a4paper,margin=0.9in]{geometry}

\usepackage[colorlinks,citecolor=magenta,linkcolor=black]{hyperref}
\pdfpagewidth=\paperwidth \pdfpageheight=\paperheight
\usepackage{amsfonts,amssymb,amsthm,amsmath,eucal,tabu,url}
\usepackage{pgf}
\usepackage{array}
\usepackage{tikz-cd}
\usepackage{pstricks}
\usepackage{pstricks-add}
\usepackage{pgf,tikz}
\usetikzlibrary{automata}
\usetikzlibrary{arrows}
\usepackage{indentfirst}
\usepackage{tabularx} 
\usepackage[style=alphabetic,doi=false,isbn=false,url=false,backend=biber,giveninits=true,maxcitenames=3,maxbibnames=8,hyperref]{biblatex}
\theoremstyle{plain}
\newtheorem{thm}{Theorem}[section]
\newtheorem{theorem}[thm]{Theorem}

\newtheorem{lemma}[thm]{Lemma}
\newtheorem{proposition}[thm]{Proposition}

\newtheorem{conjecture}[thm]{Conjecture}

\theoremstyle{definition}
\newtheorem{definition}[thm]{Definition}
\newtheorem{remark}[thm]{Remark}
\newtheorem{example}[thm]{Example}

\newtheorem{thevarthm}[thm]{\varthmname}

\newenvironment{varthm*}[1]{\trivlist\item[]{\bf #1.}\it}{\endtrivlist}

\def\keywordname{{\bfseries Keywords}}%
\def\keywords#1{\par\addvspace\medskipamount{\rightskip=0pt plus1cm
\def\and{\ifhmode\unskip\nobreak\fi\ $\cdot$
}\noindent\keywordname\enspace\ignorespaces#1\par}}
\def\subclassname{{\bfseries Mathematics Subject Classification
(2020)}\enspace}
\def\subclass#1{\par\addvspace\medskipamount{\rightskip=0pt plus1cm
\def\and{\ifhmode\unskip\nobreak\fi\ $\cdot$
}\noindent\subclassname\ignorespaces#1\par}}

\title[Numerical Terao's Conjecture and Ziegler pairs]{On the Numerical Terao's Conjecture and Ziegler pairs for line arrangements}

\author{Lukas~K\"uhne}
\address{Lukas K\"uhne, Universit\"at Bielefeld, Fakult\"at f\"ur Mathematik, Bielefeld, Germany}
\email{lkuehne@math.uni-bielefeld.de}

\author{Dante Luber}
\address{Dante Luber, Berlin Mathematical School at Technische Universit\"at Berlin, Straße des 17. Juni 45, D-10623 Berlin, Germany.}
\email{dantelubermath@gmail.com}

\author{Piotr Pokora}
\address{Piotr Pokora, Department of Mathematics, University of the National Education Commission Krakow, Podchor\c a\.zych 2, PL-30-084 Krak\'ow, Poland.}
\email{piotrpkr@gmail.com, piotr.pokora@uken.krakow.pl}

\date{\today}
\begin{document}
\begin{abstract}
In this paper we present a smallest possible counterexample to the Numerical Terao's Conjecture in the class of line arrangements in the complex projective plane. Our example consists of a pair of two arrangements with $13$ lines.
Moreover, we use the newly discovered singular matroid realization spaces to construct new examples of pairs of line arrangements having the same underlying matroid but different free resolutions of the Milnor algebras.
Such rare arrangements are called Ziegler pairs in the literature.
\subclass{14N20, 13D02, 05B35, 52C35}
\end{abstract}
\maketitle
\section{Introduction}\label{sec:intro}
In the present paper we study algebraic and combinatorial properties of line arrangements that are called \emph{$m$-syzygy}. Roughly speaking, these are line arrangements in the complex plane that are characterized by the homological properties of the free resolution of the Milnor algebra that we can associate with these arrangements.
Such arrangements can equivalently be characterized via the free resolutions of their module of logarithmic derivations.
A lot of effort has been invested towards understanding some fundamental properties of such $m$-syzygy arrangements. In this paper we will focus on the following natural general problem.
\begin{center}
\textbf{Problem}: Describe homological properties of $m$-syzygy line arrangements that are encoded by the (weak) combinatorics.
\end{center}
More precisely, the above problem is motivated by a landmark paper due to Ziegler \cite{Ziegler} which describes a construction of two arrangements of lines that have isomorphic intersection lattices, but the associated Milnor algebras have different minimal free resolutions. Recall that by the intersection lattice $L(\mathcal{L})$ of a given line arrangement $\mathcal{L} \subset \mathbb{P}^{2}_{\mathbb{F}}$ with $\mathbb{F}$ being any field we mean the set of all flats, i.e., non-empty intersections of (sub)families of lines in $\mathcal{L}$ with the order defined by the reverse inclusion.
This is the lattice of flats of the underlying \emph{matroid}.
In light of our problem, it means that the shape of the minimal free resolution of the Milnor algebra associated with Ziegler's line arrangements is not determined by the intersection lattice. 

Another path of studies was initiated by Saito and Terao in 1980's, when they introduced the notion of free hyperplane arrangements. We say that a line arrangement is \emph{free} if the associated derivation module to this arrangement is free. 
In that context, we have the following intriguing conjecture.
\begin{conjecture}[Terao]
Let $\mathbb{F}$ be a fixed field and let $\mathcal{L}_{1}, \mathcal{L}_{2} \subset \mathbb{P}^{2}_{\mathbb{F}}$ be two line arrangements. Assume that $\mathcal{L}_{1}$ and $\mathcal{L}_{2}$ have isomorphic intersection lattices. Then $\mathcal{L}_{1}$ is free if and only if~$\mathcal{L}_{2}$ is free.
\end{conjecture}
In other words, Terao's conjecture asserts that the freeness of a line arrangement over a fixed field is determined by the underlying intersection lattice, i.e., the underlying matroid.
Terao's conjecture is wide open, even for line arrangements, although we know that it holds for arrangements with up to $14$ lines \cite{BarKuh}. 

However, it is well-known that not all algebraic properties of these modules are combinatorial in general. The present manuscript provides further evidence for this observation.

First, in Section \ref{sec:numerical teroa}, we attack a weaker problem. Instead of considering the strong combinatorics of line arrangements determined by intersection lattices, one can focus on the numerical data determined by the intersection points. 
\begin{definition}
Let $\mathcal{L} \subset \mathbb{P}^{2}_{\mathbb{F}}$ be an arrangement of $d$ lines. We define the \emph{weak-combinatorics} of $\mathcal{L}$, denoted by $W(\mathcal{L})$, as a vector of the form
$$W(\mathcal{L}) = (d;t_{2}, \ldots, t_{d}),$$
where $t_{i}$ denotes the number of $i$-fold points, i.e., points where exactly $d$ lines from the arrangement meet. 
We use the convention that we truncate the vector $W(\mathcal{L})$ by removing data $t_{i} = 0$ for $i > m(\mathcal{L})$, where $m(\mathcal{L})$ denotes the maximal multiplicity of intersection points in $\mathcal{L}$. 
\end{definition}
For a given line arrangement $\mathcal{L}$, it is clear that $W(\mathcal{L})$ contains less information than the intersection lattice $L(\mathcal{L})$, but it is still interesting to ask whether the freeness can be potentially determined by the weak-combinatorics.

\begin{conjecture}[Numerical Terao's Conjecture]\label{conj:numerical_terao}
Let $\mathcal{L}_{1}, \mathcal{L}_{2} \subset \mathbb{P}^{2}_{\mathbb{F}}$ be two line arrangements such that
$W(\mathcal{L}_{1}) = W(\mathcal{L}_{2})$. Then $\mathcal{L}_{1}$ is free if and only if $\mathcal{L}_{2}$ is free.
\end{conjecture}

The Numerical Terao's Conjecture was posed, and shown to be false, in \cite{Mar}. The authors construct two complex projective line arrangements with the same weak-combinatorics, but one is free and the other is not \cite[Theorem 6.2]{Mar}.
\begin{theorem}[Marchesi-Vall\'es]
\label{pair}
There exists a pair of line arrangements in the complex projective plane, each of which has the weak-combinatorics
$$(d_{1};t_{2},t_{3},t_{4},t_{5},t_{6}) = (15; 24 ,12, 0, 0, 3),$$
such that one is free and the second is not free.
\end{theorem}
This counterexample uses the so-called triangular line arrangements, which can be seen as a modification of the full monomial arrangements. 

The first main result in this paper provides a minimal (in the sense of the number of lines) counterexample to the Numerical Terao's Conjecture.
\begin{theorem}
\label{thm:A}
Over the complex numbers, there are several counterexamples to the Numerical Terao's Conjecture for arrangements with $13$ lines.  In particular, there exists a pair of two line arrangements such that each arrangement has the weak-combinatorics
\[(d;t_{2}, t_{3}, t_{4}, t_{5}) = (13; 16, 6, 4, 2),\]
with the property that one arrangement is free and the second one is not.
Moreover, there is no counterexample to the Numerical Terao Conjecture for arrangements with $d\le 12$ lines or $d=14$ lines.
\end{theorem}

In Section \ref{sec:weak ziegler} we focus our attention on new examples of \emph{weak Ziegler pairs}. This notion has been recently introduced in \cite{Pokora1} and, roughly speaking, by a weak Ziegler pair we mean two line arrangements having the same weak-combinatorics, but they have different minimal degrees of non-trivial Jacobian relations. It is thous a weakening of pairs violating the Numerical Terao's conjecture. Here we construct a weak Ziegler pair in the class of line arrangements with only double and triple intersection points.

\begin{theorem}[{see Theorem \ref{WZ}}]
There exists a pair of (non-isomorphic) real arrangements of $12$ lines with $19$ triple and $9$ double intersection points such that they have different minimal degrees of Jacobian relations.
\end{theorem}
It is worth noting that these two line arrangements were (re)discovered by the third author and Bokowski in \cite{BokPok}, and these examples show up naturally in studies devoted to a generalized orchard problem for pseudo-line arrangements.

Section \ref{sec: ziegler pairs} devoted to \emph{Ziegler pairs}.
Here by a Ziegler pair we understand a pair of line arrangements having isomorphic intersection lattices but having different minimal degrees of non-trivial Jacobian relations.
Recently, DiPasquale, Sidman, and Traves constructed such pairs using techniques from rigidity theory~\cite{DST23}.
Our approach to construct such pairs is based on the geometry of moduli spaces of hyperplane arrangements with a fixed intersection lattice. These moduli are known as \emph{matroid realization spaces}. 
More specifically, we exploit the recently discovered explicit matroids on $12$ elements exhibiting singularities in their realization spaces~\cite{corey2023singularmatroidrealizationspaces}.
This is a novel situation since, up to now, all examples of Ziegler pairs have been constructed using line arrangements with a smooth moduli spaces.
We then obtain the following result.
\begin{theorem}[{see Theorem \ref{ZZ}}]
There exists a pair of two line arrangements $\mathcal{L}_{1}, \mathcal{L}_{2} \subset \mathbb{P}^{2}_{\mathbb{C}}$ having isomorphic intersection lattices with
$$(d;t_{2}, t_{3}) = (12;24, 14) $$
that have different minimal degrees of Jacobian relations.
\end{theorem}
We also construct another Ziegler pair consisting of two arrangements having $12$ lines with isomorphic intersection lattices, but having different minimal free resolutions of the associated Milnor algebras but equal minial degrees of their Jacobian relations.

\section*{Computations}
Some of our main results were obtained using software computations. These results can be reproduced and verified using the content of this \texttt{github} repository.
\begin{center}
    \url{https://github.com/danteluber/numerical_terao_ziegler_pairs}
\end{center}
The repository also contains the query strings and results used in the proof of Theorem~\ref{thm:A}.

\section*{Acknowledgments}
Lukas Kühne is supported by the Deutsche Forschungsgemeinschaft (DFG, German Research Foundation) SFB-TRR \textbf{358/1 2023 -- 491392403} and SPP \textbf{2458 -- 539866293}.

Dante Luber has been supported by the Deutsche Forschungsgemeinschaft (DFG, German Research Foundation) GRK \textbf{2434} -- \textbf{385256563}.

Piotr Pokora is supported by the National Science Centre (Poland) Sonata Bis Grant  \textbf{2023/50/E/ST1/00025}. For the purpose of Open Access, the authors have applied a CC-BY public copyright license to any Author Accepted Manuscript (AAM) version arising from this submission.

\section{Preliminaries}\label{sec:prelim}
In this section we introduce derivation modules, Milnor algebras, and review the necessary concepts from the theory of matroid realization. While our work focuses on hyperplane arrangements, the relevant definitions can be stated for more general curves.
\subsection{Derivation modules} Let $S := \mathbb{C}[x,y,z] = \bigoplus_{k}S_{k}$ be the graded polynomial ring and let $f \in S$ be a homogeneous polynomial of degree $d$. Consider a reduced curve $C \, : \, f=0$ in $\mathbb{P}^{2}_{\mathbb{C}}$ defined by $f$. We denote by $\partial_{x}, \partial_{y}, \partial_{z}$ the partial derivatives and we denote by $J_{f}$ the Jacobian ideal of $f$, i.e., the ideal generated by the partials $\partial_{x} f, \, \partial_{y} f, \, \partial_{z}f$.
\begin{definition}
\label{hom}
Let $C : f=0$ be a reduced curve in $\mathbb{P}^{2}_{\mathbb{C}}$ of degree $d$ given by $f \in S$. Denote by $M(f) := S/ J_{f}$ the Milnor algebra. We say that $C$ is \emph{$m$-syzygy} when $M(f)$ has the following minimal graded free resolution:
$$0 \rightarrow \bigoplus_{i=1}^{m-2}S(-e_{i}) \rightarrow \bigoplus_{i=1}^{m}S(1-d - d_{i}) \rightarrow S^{3}(1-d)\rightarrow S$$
with $e_{1} \leq e_{2} \leq \ldots \leq e_{m-2}$ and $1\leq d_{1} \leq \ldots \leq d_{m}$. 
\end{definition}
\begin{definition}
The $m$-tuple $(d_{1}, \ldots, d_{m})$ in Definition \ref{hom} is called the \emph{exponents} of $C$.
\end{definition}
In light of Definition \ref{hom}, being an $m$-syzygy curve is a homological condition that is encoded by the shape of the minimal free resolution of the Milnor algebra. Among $m$-syzygy curves, the most crucial objects, from the viewpoint of our purposes, are the following.
\begin{definition}\label{def:free}
We say that a reduced curve $C \subset \mathbb{P}^{2}_{\mathbb{C}}$ of degree $d$ is \emph{free} if and only if $C$ is $2$-syzygy.
In this case it follows that $d_{1}+d_{2}=d-1$.
\end{definition}
\begin{definition}\label{def:pog}
We say that a reduced curve $C \subset \mathbb{P}^{2}_{\mathbb{C}}$ of degree $d$ is \emph{plus-one generated} if and only if $C$ is $3$-syzygy and it holds that $d_{1}+d_{2}=d$.
\end{definition}
By the above definition, a reduced plane curve $C \subset \mathbb{P}^{2}_{\mathbb{C}}$ is free if the corresponding minimal free resolution of the Milnor algebra is described by the Hilbert--Burch Theorem.
\begin{definition}
Let $C \subset \mathbb{P}^{2}_{\mathbb{C}}$ be a reduced curve given by $f \in S_{d}$.
We define the graded $S$-module of algebraic relations associated with $f$ as
$${\rm AR}(f) = \{(a,b,c)\in S^{\oplus 3} : af_x+bf_y+cf_z=0\}.$$
Then the minimal degree of Jacobian relations among the partial derivatives of $f$ is defined~as
$${\rm mdr}(f):=\min\{r : \, {\rm AR}(f)_{r} \neq 0 \}.$$
\end{definition}
\begin{remark}
In light of Definition \ref{hom}, we have 
$${\rm mdr}(f) = d_{1}.$$
\end{remark}
\subsection{Matroid realizations}

\begin{definition}\label{def:matroid bases}
    Let $E$ be a set finite set of cardinality $n$, and  $0<d\leq n$. A rank $d$ \emph{matroid} on the ground set $E$ is the pair $M=(E,\mathcal{B}(M))$, where $\mathcal{B}(M)\subset\binom{E}{d}$ is a non-empty set system closed under the following exchange relation: for any pair $B_1,B_2\in\mathcal{B}(M)$, and $x\in B_{1}\setminus B_{2}$, there exists $y\in B_{2}\setminus B_{1}$ such that $(B_{1}\setminus x)\cup y \in\mathcal{B}(M)$. 
\end{definition}
The collection $\mathcal{B}(M)$ is the set of \emph{bases} of $M$, and a set $A\in 2^{E}$ is \emph{independent} if it is contained in some basis. In Definition \ref{def:matroid bases} and throughout, we adopt the standard practice in matroid theory to omit bracket notation in the expression of a singleton set. Most frequently, we assume $E=[n]$, and refer to rank $d$ matroids on this ground set as $(d,n)$-matroids. 

Matroids were first introduced to combinatorially abstract independence relations amongst vectors. The following class of matroids is the most perfect reflection of this historical fact.

\begin{definition}\label{def: realizable matroid hyperplanes}
A $(d,n)$-matroid $M$ is \emph{$\mathbb{F}$-realizable} if there exists an arrangement of $n$ distinct projective hyperplanes $\mathcal{A}=\{H_{1},\dots, H_{n} \} \subset \mathbb{P}^{d-1}_{\mathbb{F}}$ such that $$\mathcal{B}(M)=\left\{\lambda\in\binom{[n]}{d}:\bigcap_{i\in\lambda} H_{i}=\emptyset\right\}.$$
\end{definition}

Now suppose $M$ is $\mathbb{F}$-realizable. The following moduli space encodes all realizations of $M$ over this field, up to a suitable notion of redundancy.

\begin{definition}\label{def:realization space}
    Let $M$ be a $(d,n)$-matroid realizable over a field $\mathbb{F}$. The \emph{realization space} of $M$ over $\mathbb{F}$ is the moduli space of hyperplane arrangements $$\mathcal{R}(M;\mathbb{F}) = \{\mathcal{A}:\mathcal{A}\text{ realizes }M\}/\sim,$$ where $\sim$ denotes the action of the projective linear group $\text{PGL}(d;\mathbb{F})$ mapping one arrangement to another.
\end{definition}

    Realization spaces of matroids are semi-algebraic sets, defined by polynomial systems of equations and inequalities. From the perspective of algebraic geometry, they can be studied as affine schemes. The following result tells us they are rich sources of algebro-geometric~data.
    \begin{theorem}[Mn\"ev's universality theorem \cite{Mnev88}]
        Every singularity type appears on the $\mathbb{C}$-realization space of some $(3,n)$-matroid.
    \end{theorem}
   The constructive algorithm of Mn\"ev is useful (it was applied by Vakil to show many moduli spaces are arbitrarily singular), but yields a matroid on a large ground set. In \cite{corey2023singularmatroidrealizationspaces} the second author used \texttt{OSCAR} to obtain examples of $(3,n)$-matroids with singular realization spaces, showing that $n=12$ is minimal for matroids admitting such singularities.
   This result is the starting point in searching for new examples of Ziegler pairs that we are going to study in Section \ref{sec: ziegler pairs}.
\begin{remark}\label{rem: realizations vs. strata}
    In the literature, realization spaces refer to subvarieties in Grassmannians indexed by matroids. Elsewhere, these subvarieties are known as \emph{matroid strata} or \emph{thin Schubert cells} in Grassmannians.  In this text, we only consider moduli of hyperplane arrangements from Definition \ref{def:realization space}. However, we remark that the realization spaces we consider are closely related to these Grassmannian subvarieties via the \emph{Gelfand-MacPherson correspondence}.
\end{remark}
\subsection{Characteristic polynomials for hyperplane arrangements}
Let $\mathcal{A} \subset \mathbb{F}^{d}$ be a central hyperplane arrangement, i.e., all hyperplanes are passing through $0 \in \mathbb{F}^{d}$. This situation corresponds to the scenario that we can consider $\mathcal{A}$ as an arrangement in $\mathbb{P}^{d-1}_{\mathbb{F}}$. The intersection lattice $L(\mathcal{A})$ is the set of all subsets of the form $\cap_{H \in \mathcal{B}}H$ where $\mathcal{B} \subset \mathcal{A}$ is a subarrangement. Define $L_{i}(\mathcal{A}) := \{X \in L(\mathcal{A}) \, : \, {\rm codim}_{\mathbb{F}^{d}} \, X = i\}$. The M\"obius function $\mu : L(\mathcal{A}) \rightarrow \mathbb{Z}$ is defined by $\mu(V)=1$ and $\mu(X) = - \sum_{V \supset Y \supset X}\mu(Y)$ --- we assume also $X \neq \mathbb{F}^{d}$. 
\begin{definition}
The characteristic polynomial of $\mathcal{A}$ is defined as
$$\chi(\mathcal{A};t) = \sum_{X \in L(\mathcal{A})}\mu(X)t^{{\rm dim}\, X}.$$
\end{definition}
Since for central arrangements $\mathcal{A} \subset \mathbb{F}^{d}$ we know that $(t-1)$ is a divisor of $\chi(\mathcal{A},t)$, we can define a reduced characteristic polynomial of $\mathcal{A}$ as
$$\chi_{0}(\mathcal{A},t) = \chi(\mathcal{A},t)/(t-1).$$
\begin{remark}\label{rem:char_poly}
If $\mathcal{L} \subset \mathbb{P}^{2}_{\mathbb{F}}$ is a line arrangement having the weak-combinatorics $W(\mathcal{L}) = (d;t_{2}, \ldots, t_{d})$,  
then
$$\chi(\mathcal{L},t) = t^{3}-dt^{2} + \bigg(\sum_{r\geq 2}(r-1)t_{r}\bigg)t - \bigg(\sum_{r\geq 2}(r-1)t_{r} + 1 - d\bigg).$$
\end{remark}
Recall the following fundamental result which characterizes free line arrangements.
\begin{theorem}[Terao's factorization theorem]
If $\mathcal{L} \subset \mathbb{P}^{2}_{\mathbb{F}}$ is a free arrangement of lines, then
$$\chi_{0}(\mathcal{L},t) = (t-d_{1})(t-d_{2}),$$
where $(d_{1},d_{2})$ are the exponents of $\mathcal{L}$. In particular, this result tells us that the freeness determines $\chi_{0}(\mathcal{L},t)$.
\end{theorem}

The next lemma is a combinatorial restriction of free line arrangements with an intersection point of high multiplicity.
We will use it to deduce non-freeness below.
The proof uses the theory of multiarrangements as introduced by Ziegler~\cite{Ziegler_multi}.
We refrain from stating all definitions in this more general situation as we will not use them beyond this proof --- see~\cite{Yoshinaga} for details.
\begin{lemma}\label{lem:nonfree_arr}
	Let $\mathcal{L}$ be a free line arrangement with $n$ lines and an $k$-fold point, i.e., an intersection point where $k$ lines meet.
	Assume that $2k\ge n+1$.
	Then $k-1$ must be a root of the reduced characteristic polynomial of $\mathcal{L}$.
\end{lemma}
\begin{proof}
	Say the roots of $\chi_{0}(\mathcal{L},t)$ are $(a,b)$ for some natural numbers $a,b$.
	By the above discussion and the assumption, $\mathcal{L}$ is free with exponents $(a,b)$.
	Let $H\in \mathcal{L}$ be a line passing through the intersection point with multiplicity $k$.
	Thus, the Ziegler restriction $(\mathcal{L}^H,k^H)$ is a free $2$-multiarrangement with a line of multiplicity $k-1$ by~\cite{Ziegler_multi}.
	The condition $2k\ge n+1$ ensures that the exponents of $(\mathcal{L}^H,k^H)$ are $(k-1,n-m)$~\cite[Prop. 1.23 (i)]{Yoshinaga}.
	As $\mathcal{L}$ is assumed to be free with exponents $(a,b)$ by Ziegler's criterion, $(\mathcal{L}^H,k^H)$ has the exponents $(a,b)$.
	Therefore, $m-1$ must be a root of $\chi_{0}(\mathcal{L},t)$.
\end{proof}

\section{Numerical Terao's Conjecture}\label{sec:numerical teroa}

We start by describing the two line arrangements $\mathcal{A}$ and $\mathcal{B}$ with $13$ lines having the same weak-combinatorics, but $\mathcal{A}$ is free and $\mathcal{B}$ is not free but plus-one-generated.

\begin{example}\label{ex:ntc1}
Let $\mathcal{A}$ be the complex line arrangement given by the following defining equation
\begin{align*}
Q_{\mathcal{A}}(x,y,z)=&xyz(x+y)(x+iy)(x+(1+i)y)(x+z)(2x+(1+i)z)(x+iz)\\
&(y-z)(-2y+(1+i)z)(x+iy+z)(x+(-1+i)y+z).
\end{align*}
The weak-combinatorics of $\mathcal{A}$ is $(d_{1};t_{2}, t_{3}, t_{4}, t_{5}) = (13; 16, 6, 4, 2)$.
An \texttt{OSCAR} computation yields that $\mathcal{A}$ is free with exponents $(6,6)$.
\end{example}

\begin{example}\label{ex:ntc2}
	Let $\varphi=\frac{1+\sqrt{5}}{2}$ be the golden ratio and $\mathcal{B}$ be the real line arrangement given by the following defining equation
	\begin{align*}
	Q_{\mathcal{B}}(x,y,z)=&xyz(x+y)(x-\varphi y)(x+(1-\varphi)y)(x+z)(x-\varphi z)(x+(1+\varphi)z)\\
	&(y-z)(y-(1+\varphi)z)(x-y+z)(x+(1-\varphi)y+\varphi z).
	\end{align*}
	The weak-combinatorics of $\mathcal{B}$ is $(d_{1};t_{2}, t_{3}, t_{4}, t_{5}) = (13; 16, 6, 4, 2)$ as well.
	This arrangement is depicted in Figure~\ref{fig:arr_B}.
	
	\texttt{OSCAR} yields that module of logarithmic derivations has a minimal set of generators with degrees $(5,8,8)$ with one relation generated in degree $9$.
	Thus, $\mathcal{B}$ is plus-one-generated with exponents $(5,8,8)$.
	
	Adding the line $y+\varphi z=0$ yields a free arrangement with the exponents $(5,8)$.
	This is the red line in Figure~\ref{fig:arr_B}.
	Thus, $\mathcal{B}$ is the deletion of a free arrangement and hence plus-one-generated~\cite{Abe21}.
\end{example}
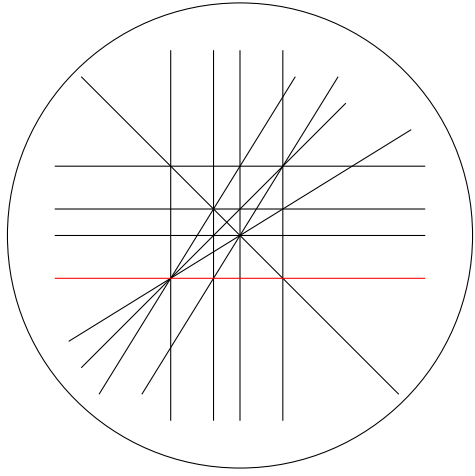
\begin{figure}
	\begin{center}
	\begin{tikzpicture}[scale=.35]
	\def\b{(1+sqrt(5))/2}
	\draw[domain=-6:6, samples=100, smooth, variable=\x, black] plot (\x, {-\x});
	\draw[domain=-7:7, samples=100, smooth, variable=\x, black] plot (0,\x);
	\draw[domain=-7:7, samples=100, smooth, variable=\x, black] plot (-1,\x);
	\draw[domain=-7:7, samples=100, smooth, variable=\x, black] plot ({\b},\x);
	\draw[domain=-7:7, samples=100, smooth, variable=\x, black] plot ( {-\b-1},\x);
	\draw[domain=-7:7, samples=100, smooth, variable=\y, black] plot (\y,0);
	\draw[domain=-7:7, samples=100, smooth, variable=\y, black] plot (\y,1);
	\draw[domain=-7:7, samples=100, smooth, variable=\y, black] plot (\y,{\b+1});
	\draw[domain=-6:6, samples=100, smooth, variable=\x, black] plot ( {(\b-1)*\x},\x);
	\draw[domain=-4:4, samples=100, smooth, variable=\x, black] plot ({\b*\x},\x);
	
	\draw[domain=-6:6, samples=100, smooth, variable=\x, black] plot ({(\b-1)*\x-\b}, {\x});
	\draw[domain=-5:5, samples=100, smooth, variable=\x, black] plot ({\x-1}, {\x});
	\draw[domain=-7:7, samples=100, smooth, variable=\y, red] plot (\y,{-\b});
	
	\draw (0,0) circle (250pt);
	\end{tikzpicture}
\end{center}
\caption{The black lines are projectivized picture of $\mathcal{B}$ with the circle the line at infinity.
Adding the red line $y+\varphi z=0$ yields a free arrangement.}
\label{fig:arr_B}
\end{figure}

These two examples prove the first part of Theorem~\ref{thm:A}, namely they provide a new counterexample to the Numerical Terao's Conjecture (NTC) on $13$ lines.
We proceed by proving the second part that says that there this no counterexample to the NTC arrangements with up to $12$ or $14$ lines.
\begin{proof}[Proof of Theorem~\ref{thm:A}]
First note that the weak-combinatorics of a line arrangement determines its characteristic polynomial, see Remark~\ref{rem:char_poly}. Assume now that the line arrangements $\mathcal{A}$ and $\mathcal{B}$ form a counterexample to the Numerical Terao's Conjecture, say $\mathcal{A}$ is free and $\mathcal{B}$ is not, but their weak-combinatorics agree.
Thus, the characteristic polynomials of these arrangements agree.
Since $\mathcal{A}$ is free, the polynomial factors over the integers by Terao's factorization theorem.

There is a database of all matroids of rank three with up to $14$ elements whose characteristic polynomial factors over the integers~\cite{Bar21}. 	Therefore, the matroids underlying the arrangements $\mathcal{A}$ and $\mathcal{B}$ appear in this database if their size is at most $14$.
We use this observation to query the database for pairs of matroids having the same weak-combinatorics, which are both realizable over $\mathbb{C}$ with at least one of them being not divisionally free.
The latter is a combinatorial, i.e., matroidal criterion~\cite{Abe16}. which ensures that all realizations are free.
We record the results of this query separately for each size of the ground set $n$:
\begin{enumerate}
	\item[($n\le 9$)] There is no pair of matroids on at most $9$ elements satisfying the above criteria.
	\item[($n=10$)] The query yields exactly one pair	of matroids both having the weak-com\-bi\-natorics $(d;t_2,t_3,t_4,t_5,t_6,t_7)=(10;21,1,0,0,0,1)$ with reduced characteristic polynomial $(t-4)(t-5)$.
	Thus, there is a $7$-fold point but $6$ is not a root of the reduced characteristic polynomial.
	Moreover, it holds that $2\cdot 7 \ge 10+1$.
	Therefore, all realizations of these two matroids are non-free by Lemma~\ref{lem:nonfree_arr}.
	\item[($n=11$)] The query yields two types of pairs of matroids on $11$ elements.
	The first type of pairs among three matroids has the weak-combinatorics \[(d;t_2,t_3,t_4,t_5,t_6,t_7)=(11;19,5,0,0,0,1)\]
	and reduced characteristic polynomial $(t-5)^2$.
	Again by Lemma~\ref{lem:nonfree_arr} all realizations of these matroids are non-free (as also $2\cdot 7 \ge 11+1$).

	Moreover, there are four matroids having the weak-combinatorics \[(d;t_2,t_3,t_4)=(11;10,5,5)\]
	with reduced characteristic polynomial $(t-5)^2$.
	All their realizations consist of two (Galois conjugate) points and we checked with \texttt{OSCAR} that they yield free arrangements in each case.
	\item[($n=12$)]
	Querying the database yields six matroids on $12$ elements, all having the  weak-combinatorics
	\[
	(d;t_2,t_3,t_4,t_5,t_6,t_7,t_8)=(12;26,4,0,0,0,0,1),
	\]
	with reduced characteristic polynomial $(t-5)(t-6)$.
	As they all have a $8$-fold point and $7$ is not a root of the reduced characteristic polynomial, we can again deduce using Lemma~\ref{lem:nonfree_arr} that all realizations are non-free.
	\item[($n=13$)]
	There are many pairs of matroids with a $k$-fold point whose reduced characteristic polynomial does not have the root $k-1$ for $k=8,9$.
	By Lemma~\ref{lem:nonfree_arr} all realizations are non-free.
	Moreover, there are several pairs of matroids with a $5$-fold point and reduced characteristic polynomial $(t-6)^2$.
	These pairs have the following distinct weak-combinatorics:
	\begin{itemize}
		\item $(d;t_2,t_3,t_4,t_5)=(13;16,6,4,2)$: There are $13$ $\mathbb{C}$-realizable matroids with this weak-combinatorics, $11$ are divisionally free and $2$ are not divisionally free.
		The two not divisionally free matroids are the ones described in the Examples~\ref{ex:ntc1} and~\ref{ex:ntc2}, thus one of them yields free realizations and the other one yields non-free realizations which is a counterexample to the NTC.
		Moreover, picking the matroid with the non-free realization and any of the divisionally free matroids yields $11$ further counterexamples to the NTC.
		\item $(d;t_2,t_3,t_4,t_5)=(13;12,12,0,3)$: There are $2$ $\mathbb{C}$-realizable matroids with this weak-combinatorics, $1$ divisionally free and $1$ not divisionally free.
		The one which is not divisionally free again yields a non-free realization.
		Thus, this is another counterexample to the NTC.
		\item $(d;t_2,t_3,t_4,t_5)=(13;14,6,6,1)$: There are $9$ $\mathbb{C}$-realizable matroids with this weak-combinatorics, $8$ divisionally free and $1$ not divisionally free.
		The realization space of the matroid which is not divisionally free consists of two points which both yield free arrangements.
		Thus, this is no counterexample to the NTC.
		\item $(d;t_2,t_3,t_4,t_5)=(13;15,9,1,3)$: There are $3$ $\mathbb{C}$-realizable matroids with this weak-combinatorics, $2$ divisionally free and $1$ not divisionally free.
		The one which is not divisionally free indeed yields a non-free realization.
		Hence, this yields two more pairs of matroids which form counterexamples to the NTC.
	\end{itemize}
	\item[($n=14$)] All appearing pairs of matroids have a $k$-fold point for $k=9,10$ but their reduced characteristic polynomial does not have the root $k-1$.
	Thus, by Lemma~\ref{lem:nonfree_arr} all realizations are again non-free.
	Hence, there is no counterexample to the NTC with $14$ lines.\qedhere
\end{enumerate}
\end{proof}

\section{Weak Ziegler pairs}\label{sec:weak ziegler}
Let us start our discussion with the following crucial definition introduced by Ziegler in~\cite{Ziegler}.
\begin{definition}[Ziegler pair]
We say that two line arrangements $\mathcal{L}_{1}, \mathcal{L}_{2} \subset \mathbb{P}^{2}_{\mathbb{C}}$ form a \emph{Ziegler pair} if the intersection lattices of $\mathcal{L}_{1}$ and $\mathcal{L}_{2}$ are isomorphic, but the minimal free resolutions of the associated Milnor algebras are different. In particular, a pair of line arrangements $\mathcal{L}_{1}, \mathcal{L}_{2} \subset \mathbb{P}^{2}_{\mathbb{C}}$ having isomorphic intersection lattices such that ${\rm mdr}(\mathcal{L}_{1}) \neq {\rm mdr}(\mathcal{L}_{2})$ forms a Ziegler pair.
\end{definition}
Ziegler pairs have a strong connection with presumably existing counterexamples to the classical Terao's freeness conjecture, since if we could find a counterexample to this conjecture, it would be all about the minimal degrees of non-trivial Jacobian relations being different \cite[Conjecture 3.5]{dimcafreecurves}.

The first pair of line arrangements with the same intersection posets but different minimal degrees of non-trivial Jacobian relations was constructed by Ziegler \cite{Ziegler}. His pair consists of two arrangements having exactly $9$ lines in the real projective plane with $6$ triple and $18$ double points as intersections, but the geometries of these arrangements are completely different. In the first case, all six triple points lie on a smooth conic, but in the second case, only $5$ triple points lie on a conic, and one point is off the conic. Using the language of linear systems of curves, $6$ points on a smooth conic impose \emph{dependent conditions} and this is the reason why this is a quite special situation. For more about this Ziegler pair and other pairs constructed using this example we refer to \cite{dimcafreecurves}.

In order to understand how weak-combinatorics can reflect on algebraic properties of plane curves, the third author introduced the following notion.

\begin{definition}[Weak Ziegler pair]
Consider two reduced plane curves $C_{1}, C_{2} \subset \mathbb{P}^{2}_{\mathbb{C}}$ such that all irreducible components of $C_{1}$ and $C_{2}$ are smooth. We say that a pair $(C_{1},C_{2})$ forms a \emph{weak Ziegler pair} if $C_{1}$ and $C_{2}$ have the same weak-combinatorics, but they have different minimal degrees of non-trivial Jacobian relations, i.e., ${\rm mdr}(C_{1}) \neq {\rm mdr}(C_{2})$.
\end{definition}
Clearly, the notion of a weak Ziegler pair in the above sense is designed for reduced plane curves, i.e., plane curve arrangements, but still it is an interesting problem to study weak Ziegler pairs in the class of line arrangement. Here we would like to present an example of two real line arrangements that form a weak Ziegler pair. The story behind these two line arrangement dates back to the classical orchard problem. In \cite{BokPok}, Bokowski and the third author presented a list of $13$ rank-$3$ oriented matroids consisting of $12$ pseudolines, $19$ triple and $9$ double intersections. It turned out that only $3$ among these $13$ oriented matroids can be realized over the real numbers as line arrangements, i.e., only $3$ of them are stretchable. It turns out that one arrangement is the well-known B\"or\"oczky line arrangement \cite{FuPa1984}, the second is the lesser known Zacharias arrangement \cite{Z41}, and the third seemed to be a new example. 
These three arrangements are not isomorphic, i.e., their intersection lattices are not pairwise isomorphic, but they have the same weak-combinatorics.

\begin{theorem}
\label{WZ}
There exists a pair of two real line arrangements $\mathcal{L}_{1}, \mathcal{L}_{2}$ consisting of $12$ lines with $19$ triple and $9$ double intersection points such that $\mathcal{L}_{1}$ is $3$-syzygy with ${\rm mdr}(\mathcal{L}_{1}) = 6$ and $\mathcal{L}_{2}$ is $4$-syzygy with ${\rm mdr}(\mathcal{L}_{1}) = 7$.
\end{theorem}

\begin{proof}
The arrangement $\mathcal{L}_{1}$ is the Zacharias arrangement constructed in \cite{Z41} that was rediscovered in \cite{BokPok}, which also appeared in \cite{BokPok} denoted as $\mathcal{C}_{7}$-arrangement. The arrangement $\mathcal{L}_{1}$ is given by the following defining equation
\begin{align*}
Q_{1}(x,y,z) = (y+2z)(y+3z)(y+5z)x(x-y)(x-y-4z)(x-y-6z)(x+y) \\ (x+y+4z)(x+y+6z)\bigg(y+\frac{x}{3}+\frac{10z}{3}\bigg)\bigg(y-\frac{x}{3}+\frac{10z}{3}\bigg).
\end{align*}
Using \verb}SINGULAR}, we can compute the minimal free resolution of the Milnor algebra of $\mathcal{L}_{1}$ obtaining
$$0 \rightarrow S(-21) \rightarrow S(-19) \oplus S(-18) \oplus S(-17) \rightarrow S^{3}(-11) \rightarrow S,$$
so $\mathcal{L}_{1}$ is $3$-syzygy with exponents $(6,7,8)$ and ${\rm mdr}(\mathcal{L}_{1})=6$.

The second arrangement $\mathcal{L}_{2}$ first appeared as $\mathcal{C}_{2}$ in \cite{BokPok}. The defining equation of $\mathcal{L}_{2}$ is:
\begin{multline*}
Q_{2}(x,y,z) = xy(x-y)(x+y-z)(x-z)(y-z)\bigg(x+(1-\sqrt{2})z\bigg)\bigg(y +(\sqrt{2}-2)z\bigg) \\
\bigg(x+\sqrt{2}y+(1-\sqrt{2})z\bigg)\bigg(x -\sqrt{2}z + \frac{\sqrt{2}}{2}y \bigg)\bigg( x + (\sqrt{2}+1)y - \sqrt{2}z\bigg)\bigg(x + (\sqrt{2}-1)y+(2-2\sqrt{2})z\bigg).
\end{multline*}
Using \verb}SINGULAR}, we can compute the minimal free resolution of the Milnor algebra of $\mathcal{L}_{2}$ obtaining
$$0 \rightarrow S^{2}(-20) \rightarrow S(-19) \oplus S^{3}(-18)  \rightarrow S^{3}(-11) \rightarrow S,$$
so $\mathcal{L}_{2}$ is $4$-syzygy with exponents $(7,7,7,8)$ and ${\rm mdr}(\mathcal{L}_{2})=7$.
\end{proof}

For the completeness of our presentation, we present geometric realizations of the arrangements $\mathcal{L}_{1}, \mathcal{L}_{2}$ below in Figure~\ref{movable} and~\ref{ex1}.

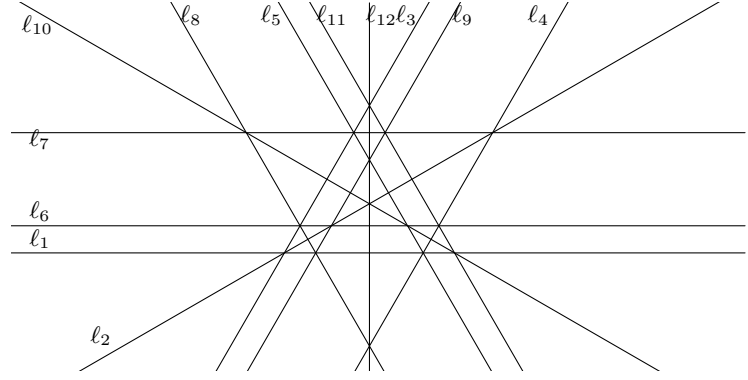
\begin{figure}[ht]
\centering
\hskip 0.5 cm
\definecolor{uuuuuu}{rgb}{0.26666666666666666,0.26666666666666666,0.26666666666666666}
\begin{tikzpicture}[line cap=round,line join=round,>=triangle 45,x=1.0cm,y=1.0cm,scale=0.45]
\clip(-11.030042187083774,-4.463352965486819) rectangle (10.538621035893197,6.347346204953949);
\draw [domain=-11.030042187083774:10.538621035893197] plot(\x,{(--10.490381056766578--4.330127018922194*\x)/2.5});
\draw [domain=-11.030042187083774:10.538621035893197] plot(\x,{(--6.160254037844389-4.330127018922194*\x)/2.5});
\draw [domain=-11.030042187083774:10.538621035893197] plot(\x,{(--5.-0.*\x)/-5.});
\draw (-0.5,-4.463352965486819) -- (-0.5,6.347346204953949);
\draw [domain=-11.030042187083774:10.538621035893197] plot(\x,{(-3.169872981077803-2.5*\x)/-4.330127018922194});
\draw [domain=-11.030042187083774:10.538621035893197] plot(\x,{(--0.669872981077809-2.5*\x)/4.330127018922194});
\draw [domain=-11.030042187083774:10.538621035893197] plot(\x,{(-6.505652582187747-4.330127018922194*\x)/-2.5});
\draw [domain=-11.030042187083774:10.538621035893197] plot(\x,{(-2.175525563265561--4.330127018922194*\x)/-2.5});
\draw [domain=-11.030042187083774:10.538621035893197] plot(\x,{(--2.331118627451304-0.*\x)/0.9202336229782615});
\draw [domain=-11.030042187083774:10.538621035893197] plot(\x,{(-0.45469932893417664-0.*\x)/2.239299131065215});
\draw [domain=-11.030042187083774:10.538621035893197] plot(\x,{(--5.854893586094159-3.533181324006426*\x)/-2.0398831885108697});
\draw [domain=-11.030042187083774:10.538621035893197] plot(\x,{(-9.388074910100586-3.533181324006429*\x)/2.0398831885108697});
\begin{scriptsize}
\draw[color=black] (0.591459421140065,5.978200379621825) node {$\ell_{3}$};
\draw[color=black] (-1.5970479719003856,5.978200379621825) node {$\ell_{11}$};
\draw[color=black] (-10.179688410872274,-0.5873217994995183) node {$\ell_{1}$};
\draw[color=black] (-0.14683222952418334,5.978200379621825) node {$\ell_{12}$};
\draw[color=black] (-8.386694402116241,-3.43501816634733) node {$\ell_{2}$};
\draw[color=black] (-10.258791087729156,5.701341010622732) node {$\ell_{10}$};
\draw[color=black] (2.2394318556584767,5.978200379621825) node {$\ell_{9}$};
\draw[color=black] (-3.37685820118027,5.991384159097972) node {$\ell_{5}$};
\draw[color=black] (-10.192872190348421,2.3394772442051766) node {$\ell_{7}$};
\draw[color=black] (-10.179688410872274,0.23007252802161265) node {$\ell_{6}$};
\draw[color=black] (4.46749058712737,5.991384159097972) node {$\ell_{4}$};
\draw[color=black] (-5.7499385068867825,5.978200379621825) node {$\ell_{8}$};
\end{scriptsize}
\end{tikzpicture}
\caption{Projectivized picture of the arrangement $\mathcal{L}_{1}$.}
\label{movable}
\end{figure}

\begin{figure}[ht]
	\centering
	\begin{tikzpicture}[line cap=round,line join=round,>=triangle 45,x=1.0cm,y=1.0cm,scale=0.45]
	\clip(-11.08909649849381,-5.311877764308603) rectangle (13.712920118917216,7.1194508581101585);
	\draw (-1.,-5.311877764308603) -- (-1.,7.1194508581101585);
	\draw (2.,-5.311877764308603) -- (2.,7.1194508581101585);
	\draw [domain=-11.08909649849381:13.712920118917216] plot(\x,{(-3.-0.*\x)/3.});
	\draw [domain=-11.08909649849381:13.712920118917216] plot(\x,{(--6.-0.*\x)/3.});
	\draw [domain=-11.08909649849381:13.712920118917216] plot(\x,{(-0.--3.*\x)/3.});
	\draw [domain=-11.08909649849381:13.712920118917216] plot(\x,{(--3.-3.*\x)/3.});
	\draw (0.24246561506204212,-5.311877764308603) -- (0.24246561506204212,7.1194508581101585);
	\draw [domain=-11.08909649849381:13.712920118917216] plot(\x,{(--2.272603154813874-0.*\x)/3.});
	\draw [domain=-11.08909649849381:13.712920118917216] plot(\x,{(--1.3313927293012882-1.242465615062042*\x)/0.5150687698759159});
	\draw [domain=-11.08909649849381:13.712920118917216] plot(\x,{(--0.42614165563666995-0.5150687698759159*\x)/1.2424656150620421});
	\draw [domain=-11.08909649849381:13.712920118917216] plot(\x,{(--10.75176620366536-4.239581524425743*\x)/3.});
	\draw [domain=-11.08909649849381:13.712920118917216] plot(\x,{(-3.512184679239616-3.*\x)/4.239581524425742});
	\begin{scriptsize}
	\draw[color=black] (-0.7119691301454923,6.270482074140097) node {$\ell_{3}$};
	\draw[color=black] (1.6226950257721715,6.2856422309967055) node {$\ell_{4}$};
	\draw[color=black] (-10.18706716552562,-0.5515885113336136) node {$\ell_{7}$};
	\draw[color=black] (-10.156746851812406,2.4046420757049938) node {$\ell_{1}$};
	\draw[color=black] (-5.1690552459883055,-4.114225372636551) node {$\ell_{5}$};
	\draw[color=black] (-5.229695873414738,6.770767250408169) node {$\ell_{2}$};
	\draw[color=black] (0.5311637320963806,6.270482074140097) node {$\ell_{9}$};
	\draw[color=black] (-10.156746851812406,1.1160287428932933) node {$\ell_{6}$};
	\draw[color=black] (-1.40933634554947,6.8162477209779935) node {$\ell_{8}$};
	\draw[color=black] (-10.12642653809919,4.284501525924417) node {$\ell_{11}$};
	\draw[color=black] (-2.561508266651694,6.361443015279747) node {$\ell_{12}$};
	\draw[color=black] (-10.12642653809919,5.936958623294716) node {$\ell_{10}$};
	\end{scriptsize}
	\end{tikzpicture}
	\caption{Projectivized picture of the arrangement $\mathcal{L}_{2}$.}
	\label{ex1}
\end{figure}
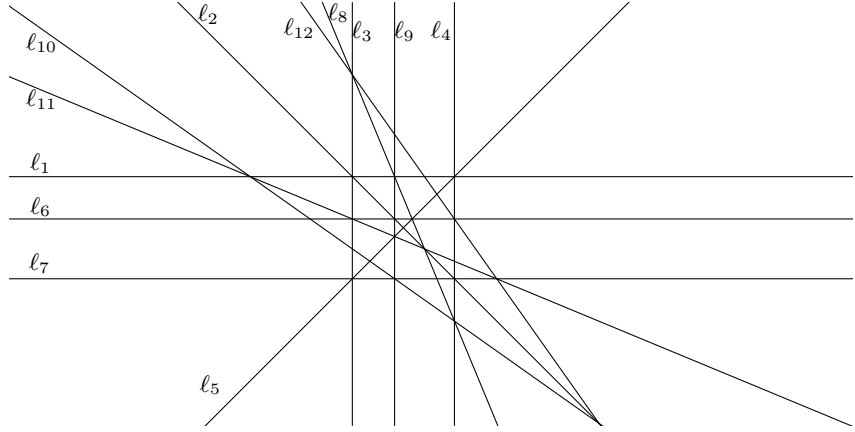

Regarding the aforementioned Zacharias arrangement, it is natural to wonder whether the realization space of that arrangement is non-trivial due to its very symmetric presentation. We have the following result.
\begin{proposition}
The moduli space of the Zacharias arrangement is one-dimensional.
\end{proposition}
\begin{proof}
Using the \texttt{OSCAR} procedure presented in \cite{CKS}, we can compute the realization space $\mathcal{R}(\mathcal{L}_{1}; \mathbb{R})$. The columns of the matrix below correspond to normal vectors of arrangements whose interaction lattice is isomorphic to that of $\mathcal{L}_1$, as we range over values for $e \in \mathbb{R} \setminus \{-2, -1, -\frac{1}{2}, 0, 1\}$. 
\begin{equation*}
    \setcounter{MaxMatrixCols}{12}
    \begin{bmatrix}
1 & 1 & 1 & 0 & 0 & e+1 & 1   &    0 & e+1 &   e+1 & e^2+2e+1  &  e^2 + 2e+1 \\
0 & 1 & 1 & 1 & 0 & e   & 0   &    1 &   e &    -e & e^2       &  e^{2}      \\
0 & 0 & 1 & 0 & 1 & e+1 & e   & -e-1 &   e & e(e+1)& e(e+1)    & 2e(e+1)
\end{bmatrix}
\end{equation*}
Thus $\mathcal{R}(\mathcal{L}_{1}; \mathbb{R})$ is one-dimensional.
\end{proof}
In order to finish our discussion in this section we look at the realization space of $\mathcal{L}_{2}$. Using \texttt{OSCAR} we get the following observation.

\begin{proposition}
The moduli space of the line arrangement $\mathcal{L}_{2}$ is zero-dimensional.
\end{proposition}
\begin{proof}
Using the same algorithm as above, we compute the realization space $\mathcal{R}(\mathcal{L}_{2}; \mathbb{R})$ as
\begin{equation*}
    \setcounter{MaxMatrixCols}{12}
    \begin{bmatrix}
1 & 1 & 1   & 0    & 1    &  1        &   1      &   1  & 1  &   1 & 0  &  0 \\
0 & 1 & e-1 & 1    & 1/2  &  e -1/2   &   e - 1/2&  e-1 & 1/2&   0 & 1  &  0 \\
0 & 1 & e-1 &2(e-1)& -e+2 &  e        &   e-1    &  0   &  e &   1 &  0 &  1
\end{bmatrix}
\end{equation*}
with $2e^2-4e+1=0$. This shows that the realization space is zero-dimensional.
\end{proof}
\section{Ziegler pairs}\label{sec: ziegler pairs}

Ziegler developed an approach to study the freeness property using supersolvable resolutions, which uses a specific realization of a given matroid \cite{Ziegler_solvable}. The first author showed that the space of non-free arrangements is a sublocus of the realization space \cite{BarKuh}. Hence, the geometry of a realization space could be an indicator of the existence of pairs of arrangements exhibiting non-uniform behavior related to freeness. 

We want to return to the original problem of Ziegler pairs, that is pairs of arrangements with the same intersection lattice but non-isomorphic Milnor algebras. In this section, we experimentally study arrangements associated $(3,12)$-matroids which have singular realization spaces. Recall that $n=12$ is the minimal value with the property that the realization space of the associated matroid admits singularities, see~\cite{corey2023singularmatroidrealizationspaces}. Our strategy for this section is the following.
We present explicitly constructed examples of arrangements which form Ziegler pairs, and keep track of where they lie on the realization space. Then we sum up our considerations in theorems accordingly.

\begin{example}\label{ex:realization_1}
    Let $M_1$ be the $(3,12)$-matroid whose non-bases are given by 
    \begin{gather*}
         \{1, 2, 3\},
 \{1, 4, 5\},
 \{1, 6, 10\},
 \{2, 4, 6\},
 \{2, 5, 8\},
 \{2, 9, 10\},
 \{3, 4, 7\},\\
 \{3, 5, 9\},
 \{3, 6, 11\},
 \{4, 10, 12\},
 \{5, 6, 7\},
 \{5, 11, 12\},
 \{7, 8, 11\},
 \{8, 9, 12\}.
    \end{gather*}
This matroid can be realized over the complex numbers and the resulting arrangements have the weak-combinatorics $(d;t_{2},t_{3}) = (12;24,14)$.

Using the algorithm in \texttt{OSCAR}, we compute on the realization space of $M_1$. The coordinate ring of $\mathcal{R}(M_1;\mathbb{C})$ is isomorphic to the ring $$R_1 = \frac{P_1^{-1}\mathbb{C}[x^{\pm},y^{\pm},z^{\pm}]}{I_1},$$
where $P_1$ is a multiplicative semigroup with $32$ generators, and $I_1$ is the principal ideal $$I = \langle(xy + xz - x - y - z^2 + 1)(x - y - z)\rangle.$$
Hence, $\mathcal{R}(M_1;\mathbb{C})$ has two maximal components, $C_1$ and $C_2$, corresponding to $xy + xz - x - y - z^2 + 1 = 0$ and $x - y - z = 0$, respectively. The singular locus of $\mathcal{R}(M_1;\mathbb{C})$ is given by the one dimensional subvariety at the intersection.
We checked that this subvariety is a smooth conic that is not excluded in the multiplicative semigroup $P_1$.

With respect to this parameterization, the line arrangements realizing $M_1$ have normal vectors of the form

\begin{equation*}
    \setcounter{MaxMatrixCols}{20}
    \begin{bmatrix}
    1 &1 &1 &1 &0 &0 &0 &x-1 &y+z-1  &y &1 &1\\
    1 &0 &x &0 &1 &0 &x &x^2-xz-x &xy &y &x &y\\
    1 &0 &x &1 &0 &1 &x-1 &0 &xy+xz-x &y+z-1 &z &z
\end{bmatrix}.
\end{equation*}

We sample points on the maximal components and in the singular locus, computing the corresponding Milnor algebras and their free resolutions.
On both $C_1$ and $C_2$, the arrangements have free resolutions of the associated Milnor algebras of the following form
$$0\to S^{2}(-21)\oplus S(-20)\to S^{5}(-19)\to S^{3}(-11) \to S.$$
Thus, on these two components the line arrangements have the exponents $(8,8,8,8,8)$.
In the singular locus, we see that the free resolutions of the associated Milnor algebras change, namely
$$0\to S^{3}(-21)\to S^{2}(-20)\oplus S^{2}(-19)\oplus S(-18)\to S^{3}(-11) \to S.$$
Hence, here the associated line arrangements have the exponents $(7,8,8,9,9)$.

\end{example}
Based on the above considerations, we notice that we have discovered a new Ziegler pair of a completely different nature compared with the previously known examples. This observation is summarized as follows.
\begin{theorem}
\label{ZZ}
There exists a pair of two degree $12$ line arrangements $\mathcal{L}_{1}, \mathcal{L}_{2}$ having isomorphic intersection lattices with the weak-combinatorics
$$(d,t_{2}, t_{3}) = (12;24,14)$$
such that ${\rm mdr}(\mathcal{L}_{1})=7$ and ${\rm mdr}(\mathcal{L}_{2}) = 8$. In other words, the line arrangements $\mathcal{L}_{1}, \mathcal{L}_{2}$ form a Ziegler pair.
\end{theorem}
Now we move on to another example of line arrangements whose intersection lattice yields the same matroid, but the algebraic properties of their free resolutions differ.
\begin{example}\label{ex:almostpog}
Let $M_2$ be the matroid whose non-bases are given by 
\begin{gather*}
\{1, 2, 3\},
 \{1, 2, 4\},
 \{1, 3, 4\},
 \{1, 5, 6\},
 \{1, 5, 7\},
 \{1, 6, 7\},
 \{1, 9, 12\},\\
 \{2, 3, 4\},
 \{2, 5, 8\},
 \{2, 5, 9\},
 \{2, 7, 11\},
 \{2, 8, 9\},
 \{3, 5, 10\},
 \{3, 6, 8\},\\
 \{4, 6, 9\},
 \{4, 7, 8\},
 \{4, 11, 12\},
 \{5, 6, 7\},
 \{5, 8, 9\},
 \{6, 10, 12\},
 \{8, 10, 11\}.
\end{gather*}
The matroid $M_2$ has realizations as complex line arrangement with the following weak-combinatorics $(d;t_{2},t_{3},t_{4}) = (12;21,9,3)$.
Using \texttt{OSCAR} we compute the realization space of~$M_2$ as $$\mathcal{R}(M_2;\mathbb{C})\cong\text{Spec}\left(\frac{P_2^{-1}\mathbb{C}[x^{\pm 1},y^{\pm1}]}{I_2}\right),$$
where $P_2$ is a multiplicative semigroup with $45$ generators, and 
    $$I_2 = \left\langle\left(x-\frac{1-i\sqrt{3}}{2}\right)\left(x-\frac{1+i\sqrt{3}}{2}\right)(xy+x-y+1)\right\rangle.$$

Therefore, $\mathcal{R}(M_2;\mathbb{C})$ has three  (one-dimensional) connected components, and a zero-dimensional singular locus where these curves intersect.
We checked that this singular locus is not excluded by the multiplicative semigroup $P_2$.
With respect to this parameterization, the normal vectors for arrangements realizing $M_2$ have the form 
\begin{equation*}
\setcounter{MaxMatrixCols}{20}
       \begin{bmatrix}
           1&1&1&0&1&1&0&0&1&\theta&1&1\\
           1&x&0&1&1&1&0&1&x+1&xy+1&x&\theta\\
           0&0&0&0&1&\gamma&1&\gamma&\gamma&xy+1&y&y\\
       \end{bmatrix},
\end{equation*}
where $\gamma = \frac{x^{2}y+x^{2}+xy-x+y+1}{2x^{2}y}$ and $\theta = -x^{2}y+xy-y+1$.

Again we sample points on the connected components and study the associated Milnor algebras. On the component corresponding to $\langle xy+x-y+1\rangle$, when $x\neq \frac{1\pm i\sqrt{3}}{2}$, the minimal free resolutions of the associated Milnor algebras have the following form
\begin{equation}\label{ex2:res1}
    0\to S^{3}(-20)\to S^{3}(-19)\oplus S^{2} (-18)\to S^{3}(-11) \to S,
\end{equation}
so the corresponding line arrangement is $5$-syzygy with exponents $(7,7,8,8,8)$. However, when $x = \frac{1\pm i\sqrt{3}}{2}$, we obtain free resolutions of the form 
\begin{equation}\label{ex2:res2}
    0\to S(-20)\to S^{3}(-18) \to S^{3}(-11) \to S.
\end{equation}
Thus, these are examples of $3$-syzygy line arrangements with exponents $(7,7,7)$.

\end{example}

Example \ref{ex:almostpog} gives us the following result.
\begin{theorem}
 There exists a pair of two degree $12$ line arrangements $\mathcal{L}_{1}, \mathcal{L}_{2}$ having isomorphic intersection lattices with the weak-combinatorics
 $(d;t_{2},t_{3},t_{4}) = (12;21,9,3)$
such that the minimal free resolutions of the associated Milnor algebras are different. In other words, the arrangements $\mathcal{L}_{1}, \mathcal{L}_{2}$ form a Ziegler pair. 
\end{theorem}

\begin{remark}
	One possible reason for the phenomena exhibited in the Examples~\ref{ex:realization_1} and~\ref{ex:almostpog} is that in both cases there are \emph{hidden collinearities} between the intersection points.
	Thus, there are collinear points in the realizations in the singular locus (in  Example~\ref{ex:realization_1}) or the linear components (in Example~\ref{ex:almostpog}) which are not collinear in other realizations.
\end{remark}
\section*{Conflict of Interests}
We declare that there is no conflict of interest regarding the publication of this paper.
\section*{Data Availability Statement}
We do not analyze or generate any datasets, because this work proceeds within a theoretical and mathematical approach. 
\printbibliography
\end{document}